\documentclass[12pt,twoside]{amsart}
\usepackage{amsmath, amsthm, amscd, amsfonts, amssymb, graphicx}

\usepackage{enumerate}
\usepackage[colorlinks=true,
            linkcolor=magenta,
            urlcolor=cyan,
            citecolor=cyan]{hyperref}

\usepackage{mathrsfs}
\usepackage{fonttable}
\newtheorem{theorem}{Theorem}[section]
\newtheorem{lemma}{Lemma}[section]
\newtheorem{remark}{Remark}[section]

\newtheorem{corollary}{Corollary}[section]

\numberwithin{equation}{section}

\begin{document}
\title{Around Jensen's inequality for strongly convex functions}
\author{Hamid Reza Moradi$^1$, Mohsen Erfanian Omidvar$^2$, Muhammad Adil Khan$^3$ and Kazimierz Nikodem$^4$}
 \maketitle
\begin{abstract}
In this paper we use basic properties of strongly convex functions to obtain new inequalities including Jensen type and Jensen-Mercer type inequalities. Applications for special means are pointed out as well. We also give a Jensen's operator inequality for strongly convex functions. As a corollary, we improve the H\"older-McCarthy inequality under suitable conditions. More precisely we show that if $Sp\left( A \right)\subset \left( 1,\infty  \right)$, then
\[{{\left\langle Ax,x \right\rangle }^{r}}\le \left\langle {{A}^{r}}x,x \right\rangle -\frac{{{r}^{2}}-r}{2}\left( \left\langle {{A}^{2}}x,x \right\rangle -{{\left\langle Ax,x \right\rangle }^{2}} \right),\quad r\ge 2\]
and if $Sp\left( A \right)\subset \left( 0,1  \right)$, then
\[\left\langle {{A}^{r}}x,x \right\rangle \le {{\left\langle Ax,x \right\rangle }^{r}}+\frac{r-{{r}^{2}}}{2}\left( {{\left\langle Ax,x \right\rangle }^{2}}-\left\langle {{A}^{2}}x,x \right\rangle  \right),\quad 0<r<1\]
for each positive operator $A$ and $x\in \mathcal{H}$ with $\left\| x \right\|=1$.
\end{abstract}
\pagestyle{myheadings}
\markboth{\centerline {Around Jensen's inequality for strongly convex functions}}
{\centerline {H.R. Moradi, M.E. Omidvar, M. Adil Khan \& K. Nikodem}}
\bigskip
\bigskip
\section{\bf Introduction and Preliminaries}
\vskip 0.4 true cm
Let $I\subset \mathbb{R}$ be an interval and $c$ be a positive number. Following Polyak \cite{6}, a function $f:I\to \mathbb{R}$ is called {\it strongly convex with modulus $c$} if
\begin{equation}\label{8}
f\left( \lambda x+\left( 1-\lambda \right)y \right)\le \lambda f\left( x \right)+\left( 1-\lambda \right)f\left( y \right)-c\lambda \left( 1-\lambda \right){{\left( x-y \right)}^{2}},
\end{equation}
for all $x,y\in I$ and $\lambda \in \left[ 0,1 \right]$. Obviously, every strongly convex function is convex. Observe also that, for instance, affine functions are not strongly convex. Since strong convexity is strengthening the notion of convexity, some properties of strongly convex functions are just “stronger versions” of known properties of convex functions. For instance, a
function $f:I\to \mathbb{R}$ is strongly convex with modulus $c$ if and only if for every ${{x}_{0}}\in \overset{o}{\mathop{I}}\,$ (the interior of $I$) there exists a number $l\in \mathbb{R}$ such that
\begin{equation}\label{41}
c{{\left( x-{{x}_{0}} \right)}^{2}}+l\left( x-{{x}_{0}} \right)+f\left( {{x}_{0}} \right)\le f\left( x \right),\quad x\in I.
\end{equation}
In other words, $f$ has a quadratic support at ${{x}_{0}}$. For differentiable functions $f$, $f$ is strongly convex with modulus $c$ if and only if
\begin{equation}\label{54}
\left( f'\left( x \right)-f'\left( y \right) \right)\left( x-y \right)\ge 2c{{\left( x-y \right)}^{2}},
\end{equation}
for each $x,y\in I$.  We recommend the book \cite{13} and the articles \cite{10,54} for more details on strongly convex functions.

A basic result concerning the convex functions is Jensen's inequality. Its formal statement is as follows: If $f$ is a convex function on an interval $\left[ m,M \right]$, then
\begin{equation*}\label{13}
f\left( \sum\limits_{i=1}^{n}{{{p}_{i}}{{x}_{i}}} \right)\le \sum\limits_{i=1}^{n}{{{p}_{i}}f\left( {{x}_{i}} \right)},
\end{equation*}
for all ${{x}_{i}}\in \left[ m,M \right]$ and all ${{p}_{i}}\in \left[ 0,1 \right]$ $\left( i=1,\ldots ,n \right)$ with $\sum\limits_{i=1}^{n}{{{p}_{i}}}=1$.

There are several inequalities which are special kinds of this inequality. So, many mathematicians paid their attention to get generalizations and reformulations of this inequality.

Consider a real valued function $f$ defined on an interval $I$, ${{x}_{1}},\ldots ,{{x}_{n}}\in I$ and ${{p}_{1}},\ldots ,{{p}_{n}}\in \left[ 0,1 \right]$ with $\sum\limits_{i=1}^{n}{{{p}_{i}}}=1$. The {\it Jensen functional} is defined by
\[{{\mathcal{J}}_{n}}\left( f,\mathbf{x},\mathbf{p} \right)=\sum\limits_{i=1}^{n}{{{p}_{i}}f\left( {{x}_{i}} \right)}-f\left( \sum\limits_{i=1}^{n}{{{p}_{i}}{{x}_{i}}} \right).\]
According to {{\cite[Theorem 1]{5}}}, if $\mathbf{x}=\left( {{x}_{1}},\ldots ,{{x}_{n}} \right)\in {{I}^{n}},\text{ }\mathbf{p}=\left( {{p}_{1}},\ldots ,{{p}_{n}} \right),\text{ }\mathbf{q}=\left( {{q}_{1}},\ldots ,{{q}_{n}} \right)$ are non-negative $n$-tuples satisfying $\sum\limits_{i=1}^{n}{{{p}_{i}}}=1,\text{ }\sum\limits_{i=1}^{n}{{{q}_{i}}}=1,\text{ }{{q}_{i}}>0,\text{ }i=1,\ldots ,n$, then
\[\underset{1\le i\le n}{\mathop{\min }}\,\left\{ \frac{{{p}_{i}}}{{{q}_{i}}} \right\}{{\mathcal{J}}_{n}}\left( f,\mathbf{x},\mathbf{q} \right)\le {{\mathcal{J}}_{n}}\left( f,\mathbf{x},\mathbf{p} \right)\le \underset{1\le i\le n}{\mathop{\max }}\,\left\{ \frac{{{p}_{i}}}{{{q}_{i}}} \right\}{{\mathcal{J}}_{n}}\left( f,\mathbf{x},\mathbf{q} \right).\]
For more results concerning Jensen's functional the reader is referred to \cite{8,9}.

In paper {{\cite[Theorem 1.2]{2}}} Mercer proved the following variant of Jensen's inequality, to which we will refer as to the {\it Jensen-Mercer's inequality}. If $f$ is a convex function on $\left[ m,M \right]$, then
\begin{equation}\label{12}
f\left( M+m-\sum\limits_{i=1}^{n}{{{p}_{i}}{{x}_{i}}} \right)\le f\left( M \right)+f\left( m \right)-\sum\limits_{i=1}^{n}{{{p}_{i}}f\left( {{x}_{i}} \right)},
\end{equation}
for all ${{x}_{i}}\in \left[ m,M \right]$ and all ${{p}_{i}}\in \left[ 0,1 \right]$ $\left( i=1,\ldots ,n \right)$ with $\sum\limits_{i=1}^{n}{{{p}_{i}}}=1$. We refer the reader to \cite{12,11,14}
as a sample of the extensive use of this inequality in this field.
\vskip 0.3 true cm
The principal aim of this article is to derive some results related to the Jensen functional in the framework of strongly convex functions (see Theorem \ref{21}). We present Jensen-Mercer's inequality for this class of functions (see Theorem \ref{5}) and give some applications for means (Corollary \ref{2.2}). In particular some new refinements of Jensen's operator inequality for strongly convex functions are also given (Theorem \ref{53} and Theorem \ref{32}).
\section{\bf On The Jensen Inequality For Strongly Convex Functions}
\vskip 0.4 true cm
The following lemma due to Merentes and Nikodem {{\cite[Theorem 4]{1}}} is the starting point for our discussion.
\begin{lemma}\label{2}
If $f:I\to \mathbb{R}$ is strongly convex with modulus $c$, then
\begin{equation}\label{20}
c\sum\limits_{i=1}^{n}{{{p}_{i}}{{\left( {{x}_{i}}-\overline{x} \right)}^{2}}}\le {{\mathcal{J}}_{n}}\left( f,\mathbf{x},\mathbf{p} \right),
\end{equation}
for all ${{x}_{1}},\ldots ,{{x}_{n}}\in I$, ${{p}_{1}},\ldots ,{{p}_{n}}>0$ with ${{p}_{1}}+\ldots +{{p}_{n}}=1$ and $\overline{x}=\sum\limits_{i=1}^{n}{{{p}_{i}}{{x}_{i}}}$.
\end{lemma}
\vskip 0.3 true cm
In {{\cite[Corollary 3]{editor}}} the following result has been given:
\begin{theorem}\label{21}
Let $f$ be a strongly convex function with modulus $c$, $\mathbf{x}=\left( {{x}_{1}},\ldots ,{{x}_{n}} \right)\in {{I}^{n}},\text{ }\mathbf{p}=\left( {{p}_{1}},\ldots ,{{p}_{n}} \right),\text{ }\mathbf{q}=\left( {{q}_{1}},\ldots ,{{q}_{n}} \right)$ non-negative $n$-tuples satisfying $\sum\limits_{i=1}^{n}{{{p}_{i}}}=1,\text{ }\sum\limits_{i=1}^{n}{{{q}_{i}}}=1,\text{ }{{q}_{i}}>0,\text{ }i=1,\ldots ,n$. Then
\[\begin{aligned}
& m{{\mathcal{J}}_{n}}\left( f,\mathbf{x},\mathbf{q} \right)+c\left( \sum\limits_{i=1}^{n}{\left( {{p}_{i}}-m{{q}_{i}} \right){{\left( {{x}_{i}}-\sum\limits_{i=1}^{n}{{{p}_{i}}{{x}_{i}}} \right)}^{2}}}+m{{\left( \sum\limits_{i=1}^{n}{\left( {{p}_{i}}-{{q}_{i}} \right){{x}_{i}}} \right)}^{2}} \right) \\
& \le {{\mathcal{J}}_{n}}\left( f,\mathbf{x},\mathbf{p} \right) \\
& \le M{{\mathcal{J}}_{n}}\left( f,\mathbf{x},\mathbf{q} \right)-c\left( \sum\limits_{i=1}^{n}{\left( M{{q}_{i}}-{{p}_{i}} \right){{\left( {{x}_{i}}-\sum\limits_{j=1}^{n}{{{q}_{j}}{{x}_{j}}} \right)}^{2}}}+{{\left( \sum\limits_{i=1}^{n}{\left( {{p}_{i}}-{{q}_{i}} \right){{x}_{i}}} \right)}^{2}} \right), \\
\end{aligned}\]
where
$$m:=\underset{1\le i\le n}{\mathop{min}}\,\left\{ \frac{{{p}_{i}}}{{{q}_{i}}} \right\},\quad M:=\underset{1\le i\le n}{\mathop{\max }}\,\left\{ \frac{{{p}_{i}}}{{{q}_{i}}} \right\}.$$
\end{theorem}
\begin{remark}
Notice that Lemma \ref{2} also holds for functions defined on open convex subsets of an inner product space (cf. {{\cite[Theorem 2]{10}}}). Therefore, Theorem \ref{21} can be also formulated and proved in such more general settings.
\end{remark}
\vskip 0.3 true cm
An interesting corollary can be deduced at this stage. We restrict ourselves to the case when $n=2$.
\begin{remark}\label{36}
 According to Hiriart-Urruty and Lemar\'echal {{\cite[Prop 1.1.2]{51}}}, the function $f:I\to \mathbb{R}$ is strongly convex with modulus $c$ if and only if the function $g:I\to \mathbb{R}$ defined by $g\left( x \right)=f\left( x \right)-c{{x}^{2}}$ is convex.
Hence the function $f:\left( 0,1 \right]\to \left[ 0,\infty  \right),\text{ }f\left( x \right)=-\ln x$ is  strongly convex with modulus $c=\frac{1}{2}$. Taking ${{p}_{1}}=\lambda ,\text{ }{{p}_{2}}=1-\lambda ,\text{ }{{q}_{1}}=\mu ,\text{ }{{q}_{2}}=1-\mu $ with $\lambda ,\mu \in \left[ 0,1 \right],\text{ }{{x}_{1}}=a,\text{ }{{x}_{2}}=b$ and taking into account that $m=\min \left\{ \frac{\lambda }{\mu },\frac{1-\lambda }{1-\mu } \right\}$ and $M=\max \left\{ \frac{\lambda }{\mu },\frac{1-\lambda }{1-\mu } \right\}$, simple algebraic manipulations show that

\[\begin{aligned}
  & {{\left( \frac{\mu a+\left( 1-\mu  \right)b}{{{a}^{\mu }}{{b}^{1-\mu }}} \right)}^{m}} \\ 
 &\quad \times \exp \left( \frac{{{\left( b-a \right)}^{2}}}{2}\left( \left( \lambda -m\mu  \right){{\left( \lambda -1 \right)}^{2}}+{{\lambda }^{2}}\left( \left( 1-\lambda  \right)-m\left( 1-\mu  \right) \right)+m{{\left( \mu -\lambda  \right)}^{2}} \right) \right) \\ 
 & \le \frac{\lambda a+\left( 1-\lambda  \right)b}{{{a}^{\lambda }}{{b}^{1-\lambda }}} \\ 
 & \le {{\left( \frac{\mu a+\left( 1-\mu  \right)b}{{{a}^{\mu }}{{b}^{1-\mu }}} \right)}^{M}} \\ 
 & \quad \times \frac{1}{\exp \left( \frac{{{\left( b-a \right)}^{2}}}{2}\left( \left( M\mu -\lambda  \right){{\left( \mu -1 \right)}^{2}}+{{\mu }^{2}}\left( M\left( 1-\mu  \right)-\left( 1-\lambda  \right) \right)+{{\left( \mu -\lambda  \right)}^{2}} \right) \right)}. \\ 
\end{aligned}\]

\end{remark}
\vskip 0.3 true cm
Remark \ref{36} admits the following important special case.
\begin{corollary}\label{34}
Let $a,b\in \left( 0,1 \right]$, then
\begin{equation*}
\begin{aligned}
  & {{K}^{r}}\left( \frac{a}{b} \right) \\ 
 &\quad  \times \exp \left( \frac{{{\left( b-a \right)}^{2}}}{2}\left( \left( \lambda -r \right){{\left( \lambda -1 \right)}^{2}}+{{\lambda }^{2}}\left( \left( 1-\lambda  \right)-r \right)+\frac{r}{2}{{\left( 1-2\lambda  \right)}^{2}} \right) \right) \\ 
 & \le \frac{\lambda a+\left( 1-\lambda  \right)b}{{{a}^{\lambda }}{{b}^{1-\lambda }}} \\ 
 & \le {{K}^{R}}\left( \frac{a}{b} \right) \\ 
 &\quad  \times \frac{1}{\exp \left( \frac{{{\left( b-a \right)}^{2}}}{8}\left( \left( R-\lambda  \right)+\left( R-\left( 1-\lambda  \right) \right)+{{\left( 1-2\lambda  \right)}^{2}} \right) \right)}, \\ 
\end{aligned}
\end{equation*}
where $r=\min \left\{ \lambda ,1-\lambda  \right\},\text{ }R=\max \left\{ \lambda ,1-\lambda  \right\}$,$\lambda \in \left[ 0,1 \right]$ and $K\left( \frac{a}{b} \right)=\frac{{{\left( a+b \right)}^{2}}}{4ab}$ is the Kantorovich constant.
\end{corollary}
\begin{proof}
The result follows from Remark \ref{36} by taking $\mu =\frac{1}{2}$.
\end{proof}
\begin{remark}
The following multiplicative refinement and reverse of the Young inequality in terms of Kantorovich's constant holds:
\begin{equation}\label{37}
{{K}^{r}}\left( \frac{a}{b} \right){{a}^{\lambda }}{{b}^{1-\lambda }}\le \lambda a+\left( 1-\lambda  \right)b\le {{K}^{R}}\left( \frac{a}{b} \right){{a}^{\lambda }}{{b}^{1-\lambda }},
\end{equation}
where $a,b>0,\text{ }\lambda \in \left[ 0,1 \right],\text{ }r=\min \left\{ \lambda ,1-\lambda  \right\}$ and $R=\max \left\{ \lambda ,1-\lambda  \right\}$.

The first inequality in \eqref{37} was obtained by Zou et al. in {{\cite[Corollary 3]{16}}} while the second one by Liao et al. {{\cite[Corollary 2.2]{17}}}.

Since $\exp \left( x \right)\ge 1$ for $x\ge 0$, Corollary \ref{34} essentially gives a refinement of the inequalities in \eqref{37}.
\end{remark}
\vskip 0.3 true cm
The key role in our proof for Theorem \ref{5} will be played by the following lemma.
\begin{lemma}\label{3}
 If $f:I\to \mathbb{R}$ is strongly convex with modulus $c$, then
\[f\left( {{x}_{1}}+{{x}_{n}}-{{x}_{i}} \right)\le f\left( {{x}_{1}} \right)+f\left( {{x}_{n}} \right)-f\left( {{x}_{i}} \right)-2c{{\lambda }_{i}}\left( 1-{{\lambda }_{i}} \right){{\left( {{x}_{1}}-{{x}_{n}} \right)}^{2}},\]
where ${{\lambda }_{i}}\in \left[ 0,1 \right]$, ${{x}_{1}}=\underset{1\le i\le n}{\mathop{\min }}\,{{x}_{i}},\text{ }{{x}_{n}}=\underset{1\le i\le n}{\mathop{\max }}\,{{x}_{i}}$ and $x_i\in I$.
\end{lemma}
\begin{proof}
We use the strategy of Mercer {{\cite[Lemma 1.3]{2}}}. Let ${{y}_{i}}={{x}_{1}}+{{x}_{n}}-{{x}_{i}},$ $i=1,2,\ldots ,n$. We may write ${{x}_{i}}={{\lambda }_{i}}{{x}_{1}}+\left( 1-{{\lambda }_{i}} \right){{x}_{n}}$ and ${{y}_{i}}=\left( 1-{{\lambda }_{i}} \right){{x}_{1}}+{{\lambda }_{i}}{{x}_{n}}$ where ${{\lambda }_{i}}\in \left[ 0,1 \right]$.

Now, using simple calculations, we obtain
\[\begin{aligned}
   f\left( {{y}_{i}} \right)&=f\left( \left( 1-{{\lambda }_{i}} \right){{x}_{1}}+{{\lambda }_{i}}{{x}_{n}} \right) \\ 
 & \le \left( 1-{{\lambda }_{i}} \right)f\left( {{x}_{1}} \right)+{{\lambda }_{i}}f\left( {{x}_{n}} \right)-c{{\lambda }_{i}}\left( 1-{{\lambda }_{i}} \right){{\left( {{x}_{1}}-{{x}_{n}} \right)}^{2}} \quad \text{(by \eqref{8})} \\ 
 & =f\left( {{x}_{1}} \right)+f\left( {{x}_{n}} \right)-\left( {{\lambda }_{i}}f\left( {{x}_{1}} \right)+\left( 1-{{\lambda }_{i}} \right)f\left( {{x}_{n}} \right) \right)-c{{\lambda }_{i}}\left( 1-{{\lambda }_{i}} \right){{\left( {{x}_{1}}-{{x}_{n}} \right)}^{2}} \\ 
 & \le f\left( {{x}_{1}} \right)+f\left( {{x}_{n}} \right)-f\left( {{\lambda }_{i}}{{x}_{1}}+\left( 1-{{\lambda }_{i}} \right){{x}_{n}} \right)-2c{{\lambda }_{i}}\left( 1-{{\lambda }_{i}} \right){{\left( {{x}_{1}}-{{x}_{n}} \right)}^{2}} \quad \text{(by \eqref{8})}\\ 
 & =f\left( {{x}_{1}} \right)+f\left( {{x}_{n}} \right)-f\left( {{x}_{i}} \right)-2c{{\lambda }_{i}}\left( 1-{{\lambda }_{i}} \right){{\left( {{x}_{1}}-{{x}_{n}} \right)}^{2}},  
\end{aligned}\]
which completes the proof.
\end{proof}
The Lemma \ref{3}  follows also from the fact that strongly convex functions are strongly Wright-convex (see, e.g., \cite{49}).
\vskip 0.3 true cm
At this point our aim is to present Jensen-Mercer's inequality for strongly convex functions.
\begin{theorem}\label{5}
Let $f:I\to \mathbb{R}$ be a strongly convex with modulus $c$, then
\[\begin{aligned}
  & f\left( {{x}_{1}}+{{x}_{n}}-\sum\limits_{i=1}^{n}{{{p}_{i}}{{x}_{i}}} \right) \\ 
 & \le f\left( {{x}_{1}} \right)+f\left( {{x}_{n}} \right)-\sum\limits_{i=1}^{n}{{{p}_{i}}f\left( {{x}_{i}} \right)}-c\left( 2\sum\limits_{i=1}^{n}{{{p}_{i}}{{\lambda }_{i}}\left( 1-{{\lambda }_{i}} \right){{\left( {{x}_{1}}-{{x}_{n}} \right)}^{2}}}+\sum\limits_{i=1}^{n}{{{p}_{i}}{{\left( {{x}_{i}}-\sum\limits_{i=1}^{n}{{{p}_{i}}{{x}_{i}}} \right)}^{2}}} \right), \\ 
\end{aligned}\]
where $\sum\limits_{i=1}^{n}{{{p}_{i}}}=1,\text{ }{{\lambda }_{i}}\in \left[ 0,1 \right],\text{ }{{x}_{1}}=\underset{1\le i\le n}{\mathop{\min }}\,{{x}_{i}},\text{ }{{x}_{n}}=\underset{1\le i\le n}{\mathop{\max }}\,{{x}_{i}}$ and $x_i\in I$.
\end{theorem}
\begin{proof}
A straightforward computation gives that
\[\begin{aligned}
  & f\left( {{x}_{1}}+{{x}_{n}}-\sum\limits_{i=1}^{n}{{{p}_{i}}{{x}_{i}}} \right) \\ 
 & =f\left( \sum\limits_{i=1}^{n}{{{p}_{i}}\left( {{x}_{1}}+{{x}_{n}}-{{x}_{i}} \right)} \right) \\ 
 & \le \sum\limits_{i=1}^{n}{{{p}_{i}}f\left( {{x}_{1}}+{{x}_{n}}-{{x}_{i}} \right)}-c\sum\limits_{i=1}^{n}{{{p}_{i}}{{\left( {{x}_{i}}-\sum\limits_{i=1}^{n}{{{p}_{i}}{{x}_{i}}} \right)}^{2}}} \quad \text{(by Lemma \ref{2})}\\ 
 & \le f\left( {{x}_{1}} \right)+f\left( {{x}_{n}} \right)-\sum\limits_{i=1}^{n}{{{p}_{i}}f\left( {{x}_{i}} \right)} \\ 
 &\quad  -c\left( 2\sum\limits_{i=1}^{n}{{{p}_{i}}{{\lambda }_{i}}\left( 1-{{\lambda }_{i}} \right){{\left( {{x}_{1}}-{{x}_{n}} \right)}^{2}}}+\sum\limits_{i=1}^{n}{{{p}_{i}}{{\left( {{x}_{i}}-\sum\limits_{i=1}^{n}{{{p}_{i}}{{x}_{i}}} \right)}^{2}}} \right) \quad \text{(by Lemma \ref{3})},\\ 
\end{aligned}\]
as required.
\end{proof}
\begin{remark}\label{31}
Based on Theorem \ref{5}, we obtain that
\begin{equation*}
\begin{aligned}
  & f\left( {{x}_{1}}+{{x}_{n}}-\sum\limits_{i=1}^{n}{{{p}_{i}}{{x}_{i}}} \right) \\ 
 & \le f\left( {{x}_{1}} \right)+f\left( {{x}_{n}} \right)-\sum\limits_{i=1}^{n}{{{p}_{i}}f\left( {{x}_{i}} \right)}-c\left( 2\sum\limits_{i=1}^{n}{{{p}_{i}}{{\lambda }_{i}}\left( 1-{{\lambda }_{i}} \right){{\left( {{x}_{1}}-{{x}_{n}} \right)}^{2}}}+\sum\limits_{i=1}^{n}{{{p}_{i}}{{\left( {{x}_{i}}-\sum\limits_{i=1}^{n}{{{p}_{i}}{{x}_{i}}} \right)}^{2}}} \right) \\ 
 & \le f\left( {{x}_{1}} \right)+f\left( {{x}_{n}} \right)-\sum\limits_{i=1}^{n}{{{p}_{i}}f\left( {{x}_{i}} \right)}. \\ 
\end{aligned}
\end{equation*}
\end{remark}
\begin{corollary}\label{2.2}
Let us now define
\[\widetilde{\mathscr{A}}:={{x}_{1}}+{{x}_{n}}-\mathscr{A},\qquad \widetilde{\mathscr{G}}:=\frac{{{x}_{1}}{{x}_{n}}}{\mathscr{G}},\]
where $\mathscr{A}$ and $\mathscr{G}$ denote the usual arithmetic and geometric means respectively.

As mentioned above, the function $f:\left( 0,1 \right]\to \left[ 0,\infty  \right),\text{ }f\left( x \right)=-\ln x$ is strongly convex with modulus $c=\frac{1}{2}$. From Remark \ref{31} we obtain
\[-\ln \widetilde{\mathscr{A}}\le -\ln \widetilde{\mathscr{G}}-\frac{1}{2}\mathsf{M}\le -\ln \widetilde{\mathscr{G}},\]
where $\mathsf{M}=2\sum\limits_{i=1}^{n}{{{p}_{i}}{{\lambda }_{i}}\left( 1-{{\lambda }_{i}} \right){{\left( {{x}_{1}}-{{x}_{n}} \right)}^{2}}}+\sum\limits_{i=1}^{n}{{{p}_{i}}{{\left( {{x}_{i}}-\sum\limits_{i=1}^{n}{{{p}_{i}}{{x}_{i}}} \right)}^{2}}}$. Hence
\begin{equation}\label{16}
\widetilde{\mathscr{G}}\le {{e}^{\frac{1}{2}\mathsf{M}}}\widetilde{\mathscr{G}}\le \widetilde{\mathscr{A}}.
\end{equation}
\end{corollary}
\begin{remark}
The inequality \eqref{16} is better than an inequality given in {{\cite[Lemma 1.1]{4}}}.
\end{remark}
\section{\bf Jensen Operator Inequality For Strongly Convex Functions}
\vskip 0.4 true cm
Mond and Pe\v cari\'c \cite{15} gave an operator extension of the Jensen inequality as follows:
\begin{theorem}\label{80}
Let $A\in \mathcal{B}\left( \mathcal{H} \right)$ be a self-adjoint operator with $Sp\left( A \right)\subset \left[ m,M \right]$ for some scalars $m<M$. If $f\left( t \right)$ is a convex function on $\left[ m,M \right]$, then  
\begin{equation}\label{40}
f\left( \left\langle Ax,x \right\rangle \right)\le \left\langle f\left( A \right)x,x \right\rangle ,
\end{equation}
for any unit vector $x\in \mathcal{H}$.
\end{theorem}
The celebrated H\"older-McCarthy inequality \cite{57} which is a special case of Theorem \ref{80} asserts that:
\begin{theorem}\label{69}
Let $A$ be a positive operator on $\mathcal{H}$ . If $x\in \mathcal{H}$ is a unit vector, then
\begin{enumerate}[(a)]
\item ${{\left\langle Ax,x \right\rangle }^{r}}\le \left\langle {{A}^{r}}x,x \right\rangle $ for all $r>1$.
\item ${{\left\langle Ax,x \right\rangle }^{r}}\ge \left\langle {{A}^{r}}x,x \right\rangle $ for all $0<r<1$.
\end{enumerate}
\end{theorem}
\vskip 0.3 true cm
Here we give a more precise estimation than inequality \eqref{40} for strongly convex functions with modulus $c$ as follows:
\begin{theorem}\label{53}
(Jensen operator inequality for strongly convex functions) Let $f:I\to \mathbb{R}$ be strongly convex with modulus $c$ and differentiable on $\overset{o}{\mathop{I}}\,$. If $A$ is a self-adjoint operator on the Hilbert space $\mathcal{H}$ with $Sp\left( A \right)\subset \overset{o}{\mathop{I}}\,$, then
\begin{equation}\label{70}
f\left( \left\langle Ax,x \right\rangle \right)\le \left\langle f\left( A \right)x,x \right\rangle -c\left( \left\langle {{A}^{2}}x,x \right\rangle -{{\left\langle Ax,x \right\rangle }^{2}} \right),
\end{equation}
for each $x\in \mathcal{H}$, with $\left\| x \right\|=1$.
\end{theorem}
\begin{proof}
It follows from \eqref{41} by utilizing functional calculus that
\begin{equation}\label{62}
c\left( {{A}^{2}}+x_{0}^{2}I-2{{x}_{0}}A \right)+lA-l{{x}_{0}}I+f\left( {{x}_{0}} \right)I\le f\left( A \right),
\end{equation}
which is equivalent to
\begin{equation}\label{52}
c\left( \left\langle {{A}^{2}}x,x \right\rangle +x_{0}^{2}-2{{x}_{0}}\left\langle Ax,x \right\rangle \right)+l\left\langle Ax,x \right\rangle -l{{x}_{0}}+f\left( {{x}_{0}} \right)\le \left\langle f\left( A \right)x,x \right\rangle,
\end{equation}
for each $x\in \mathcal{H}$, with $\left\| x \right\|=1$.

Now, by applying \eqref{52} for ${{x}_{0}}=\left\langle Ax,x \right\rangle $, we deduce the desired inequality \eqref{70}.
\end{proof}
\begin{remark}
Notice that if $A$ is positive, then the quantity $\left\langle {{A}^{2}}x,x \right\rangle -{{\left\langle Ax,x \right\rangle }^{2}}$ is positive. Therefore we have
\[f\left( \left\langle Ax,x \right\rangle \right)\le \left\langle f\left( A \right)x,x \right\rangle -c\left( \left\langle {{A}^{2}}x,x \right\rangle -{{\left\langle Ax,x \right\rangle }^{2}} \right)\le \left\langle f\left( A \right)x,x \right\rangle. \]
\end{remark}
\begin{corollary}\label{74}
(Applications for H\"older-McCarthy's inequality)

\begin{itemize}
\item Consider the function $f:\left( 1,\infty  \right)\to \mathbb{R}$, $f\left( x \right)={{x}^{r}}$ with $r\ge 2$.  It can be easily verified that this function is strongly convex with modulus $c=\frac{{{r}^{2}}-r}{2}$. Based on this fact, from Theorem \ref{53} we obtain
\begin{equation}\label{71}
{{\left\langle Ax,x \right\rangle }^{r}}\le \left\langle {{A}^{r}}x,x \right\rangle -\frac{{{r}^{2}}-r}{2}\left( \left\langle {{A}^{2}}x,x \right\rangle -{{\left\langle Ax,x \right\rangle }^{2}} \right),
\end{equation}
for each positive operator $A$ with $Sp\left( A\right) \subset \left( 1,\infty \right) $ and $x\in \mathcal{H}$ with $\left\| x \right\|=1$.
\end{itemize}
It is obvious that the inequality \eqref{71} is a refinement of Theorem \ref{69} (a).
\begin{itemize}
\item It is readily checked that the function $f:\left( 0,1 \right)\to \mathbb{R}$, $f\left( x \right)=-{{x}^{r}}$ with $0<r<1$ is a strongly convex function with modulus $c=\frac{r-{{r}^{2}}}{2}$. Similarly to the above, by using Theorem \ref{53} we get
\begin{equation}\label{72}
\left\langle {{A}^{r}}x,x \right\rangle \le {{\left\langle Ax,x \right\rangle }^{r}}+\frac{r-{{r}^{2}}}{2}\left( {{\left\langle Ax,x \right\rangle }^{2}}-\left\langle {{A}^{2}}x,x \right\rangle \right),
\end{equation}
for each positive operator $A$ with $Sp\left( A\right) \subset \left( 0,1 \right) $ and $x\in \mathcal{H}$ with $\left\| x \right\|=1$.
\end{itemize}
Apparently, inequality \eqref{72} is a refinement of Theorem \ref{69} (b).
\end{corollary}
\vskip 0.3 true cm
Now, we draw special attention to the case ${{f}^{\nu }},\text{ }\nu \in \left[ 0,1 \right]$ is strongly convex for which further refinement is possible.
\begin{theorem}\label{32}
Let $f:I\to \mathbb{R}$ be non-negative and strongly convex with modulus $c$. If ${{f}^{\nu }}$ with $\nu \in \left[ 0,1 \right]$ is strongly convex with modulus $c'$, then
\begin{equation}\label{4.3.3}
\begin{aligned}
  & f\left( \left\langle Ax,x \right\rangle  \right) \\
 & \le {{f}^{1-\nu }}\left( \left\langle Ax,x \right\rangle  \right){{f}^{\nu }}\left( \left\langle Ax,x \right\rangle  \right)+c'{{f}^{1-\nu }}\left( \left\langle Ax,x \right\rangle  \right)\left( \left\langle {{A}^{2}}x,x \right\rangle -{{\left\langle Ax,x \right\rangle }^{2}} \right) \\
 & \le {{f}^{1-\nu }}\left( \left\langle Ax,x \right\rangle  \right)\left\langle {{f}^{\nu }}\left( A \right)x,x \right\rangle  \\
 & \le {{f}^{1-\nu }}\left( \left\langle Ax,x \right\rangle  \right){{\left\langle f\left( A \right)x,x \right\rangle }^{\nu }} \\
 & \le \left( 1-\nu  \right)f\left( \left\langle Ax,x \right\rangle  \right)+\nu \left\langle f\left( A \right)x,x \right\rangle  \\
 & \le \left\langle f\left( A \right)x,x \right\rangle -c\left( 1-\nu  \right)\left( \left\langle {{A}^{2}}x,x \right\rangle -{{\left\langle Ax,x \right\rangle }^{2}} \right) \\
 & \le \left\langle f\left( A \right)x,x \right\rangle,
\end{aligned}
\end{equation}
for any positive operator $A$ with $Sp\left( A \right)\subset I$ and unit vector $x\in \mathcal{H}$.
\end{theorem}
\begin{proof}
As we have assumed above ${{f}^{\nu }}$ is strongly convex, so \eqref{70} gives
\begin{equation*}
{{f}^{\nu }}\left( \left\langle Ax,x \right\rangle \right)\le \left\langle {{f}^{\nu }}\left( A \right)x,x \right\rangle -c'\left( \left\langle {{A}^{2}}x,x \right\rangle -{{\left\langle Ax,x \right\rangle }^{2}} \right).
\end{equation*}
Multiplying both sides by ${{f}^{1-\nu }}\left( \left\langle Ax,x \right\rangle \right)$ we infer that
\[\begin{aligned}
& f\left( \left\langle Ax,x \right\rangle \right) \\
& ={{f}^{1-\nu }}\left( \left\langle Ax,x \right\rangle \right){{f}^{\nu }}\left( \left\langle Ax,x \right\rangle \right) \\
& \le {{f}^{1-\nu }}\left( \left\langle Ax,x \right\rangle \right)\left\langle {{f}^{\nu }}\left( A \right)x,x \right\rangle -c'{{f}^{1-\nu }}\left( \left\langle Ax,x \right\rangle \right)\left( \left\langle {{A}^{2}}x,x \right\rangle -{{\left\langle Ax,x \right\rangle }^{2}} \right). \\
\end{aligned}\]
Rearranging the terms, we obtain
\begin{equation}\label{3.4.1}
\begin{aligned}
& f\left( \left\langle Ax,x \right\rangle \right) \\
& \le {{f}^{1-\nu }}\left( \left\langle Ax,x \right\rangle \right){{f}^{\nu }}\left( \left\langle Ax,x \right\rangle \right)+c'{{f}^{1-\nu }}\left( \left\langle Ax,x \right\rangle \right)\left( \left\langle {{A}^{2}}x,x \right\rangle -{{\left\langle Ax,x \right\rangle }^{2}} \right) \\
& \le {{f}^{1-\nu }}\left( \left\langle Ax,x \right\rangle \right)\left\langle {{f}^{\nu }}\left( A \right)x,x \right\rangle. \\
\end{aligned}
\end{equation}
On the other hand, by using H\"older-McCarthy inequality we have
\begin{equation*}
\left\langle {{f}^{\nu }}\left( A \right)x,x \right\rangle \le {{\left\langle f\left( A \right)x,x \right\rangle }^{\nu }}.
\end{equation*}
Multiplying both sides by ${{f}^{1-\nu }}\left( \left\langle Ax,x \right\rangle \right)$ we obtain
\begin{equation}\label{3.4.2}
\begin{aligned}
   {{f}^{1-\nu }}\left( \left\langle Ax,x \right\rangle  \right)\left\langle {{f}^{\nu }}\left( A \right)x,x \right\rangle &\le {{f}^{1-\nu }}\left( \left\langle Ax,x \right\rangle  \right){{\left\langle f\left( A \right)x,x \right\rangle }^{\nu }} \\
 & \le \left( 1-\nu  \right)f\left( \left\langle Ax,x \right\rangle  \right)+\nu \left\langle f\left( A \right)x,x \right\rangle \quad \text{(by Young's inequality)} \\
 & \le \left\langle f\left( A \right)x,x \right\rangle -c \left( 1-\nu  \right) \left( \left\langle {{A}^{2}}x,x \right\rangle -{{\left\langle Ax,x \right\rangle }^{2}} \right)  \quad \text{(by \eqref{70})}.
\end{aligned}
\end{equation}
Combining \eqref{3.4.1} and \eqref{3.4.2} yields the desired result \eqref{4.3.3}.
\end{proof}
\vskip 0.3 true cm
By \eqref{54} and in a manner similar to the proof of Theorem \ref{53}, we have the following additive reverse:
\begin{theorem}\label{56}
Let $f:I\to \mathbb{R}$ be strongly convex with modulus $c$ and differentiable on $\overset{o}{\mathop{I}}\,$ whose derivative $f'$ is continuous on $\overset{o}{\mathop{I}}\,$. If $A$ is a self-adjoint operator on the Hilbert space $\mathcal{H}$ with $Sp\left( A \right)\subset \overset{o}{\mathop{I}}\,$, then
\[\left\langle {{A}^{2}}x,x \right\rangle -{{\left\langle Ax,x \right\rangle }^{2}}\le \frac{1}{2c}\left( \left\langle f'\left( A \right)Ax,x \right\rangle -\left\langle Ax,x \right\rangle \left\langle f'\left( A \right)x,x \right\rangle \right),\]
for each $x\in \mathcal{H}$, with $\left\| x \right\|=1$.
\end{theorem}
\section{\bf Further Generalization}
\vskip 0.4 true cm
By replacing $c{{\left( x-y \right)}^{2}}$ with a non-negative real valued function $F\left( x-y \right)$, we can define {\it $F$-strongly convex functions} as follows:
\begin{equation}\label{59}
f\left( \lambda x+\left( 1-\lambda \right)y \right)\le \lambda f\left( x \right)+\left( 1-\lambda \right)f\left( y \right)-\lambda \left( 1-\lambda \right)F\left( x-y \right),
\end{equation}
for each $\lambda \in \left[ 0,1 \right]$ and $x,y\in I$. (Very recently, this approach has been investigated by Adamek in \cite{53}).

We should note that, if $F$ is {\it $F$-strongly affine} (i.e. ''$=$'' instead of ''$\le $'' in \eqref{59}) then the function $f$ is $F$-strongly convex if and only if $g=f-F$ is convex (see {{\cite[Lemma 4]{53}}}).

From \eqref{59} we infer that
\[f\left( \lambda \left( x-y \right)+y \right)-f\left( y \right)+\lambda \left( 1-\lambda  \right)F\left( x-y \right)\le \lambda \left( f\left( x \right)-f\left( y \right) \right).\]
By dividing both sides by $\lambda $ we obtain
\[\frac{f\left( \lambda \left( x-y \right)+y \right)-f\left( y \right)}{\lambda }+\left( 1-\lambda  \right)F\left( x-y \right)\le f\left( x \right)-f\left( y \right).\]
Notice that if $f$ is differentiable, then by letting $\lambda \to 0$ we find that
\begin{equation}\label{60}
f'\left( y \right)\left( x-y \right)+F\left( x-y \right)+f\left( y \right)\le f\left( x \right),
\end{equation}
for all $x,y\in I$ and $\lambda \in \left[ 0,1 \right]$.
\vskip 0.3 true cm
In a manner similar to the proof of Theorem \ref{53}, if $F$ is a continuous function it follows from \eqref{60} that
\begin{equation}\label{75}
f\left( \left\langle Ax,x \right\rangle \right)\le \left\langle f\left( A \right)x,x \right\rangle -\left\langle F\left( A-\left\langle Ax,x \right\rangle \right)x,x \right\rangle,
\end{equation}
for any self-adjoint operator $A$ and $x\in \mathcal{H}$, with $\left\| x \right\|=1$.

Inequality \eqref{75} in a weaker form was obtained by Kian in {{\cite[Theorem 2.1]{67}}} for superquadratic functions.
\vskip 0.3 true cm

The following theorem is a generalization of \eqref{75}. The idea of the proof, given below for completion, is similar to that in {{\cite[Lemma 2.3]{58}}}.
\begin{theorem}\label{76}
Let $f:I\to \mathbb{R}$ be an $F$-strongly convex and differentiable function on $\overset{o}{\mathop{I}}\,$ and let $F:I\to \left[ 0,\infty  \right)$ be a continuous function. If $A$ is a self-adjoint operator on the Hilbert space $\mathcal{H}$ with $Sp\left( A \right)\subset \overset{o}{\mathop{I}}\,$ and $f\left( 0 \right)\le 0$, then

\[f\left( \left\langle Ax,x \right\rangle  \right)\le \left\langle f\left( A \right)x,x \right\rangle -\left\langle F\left( A-\frac{1}{{{\left\| x \right\|}^{2}}}\left\langle Ax,x \right\rangle  \right)x,x \right\rangle +\left( {{\left\| x \right\|}^{4}}-{{\left\| x \right\|}^{2}} \right)F\left( \frac{1}{{{\left\| x \right\|}^{2}}}\left\langle Ax,x \right\rangle  \right),\]
for each $x\in \mathcal{H}$ and $\left\| x \right\|\le 1$.
\end{theorem}
\begin{proof}
Let $y=\frac{x}{\left\| x \right\|}$ so that $\left\| y \right\|=1$. Whence
\[\begin{aligned}
   f\left( \left\langle Ax,x \right\rangle  \right)&=f\left( {{\left\| x \right\|}^{2}}\left\langle Ay,y \right\rangle +\left( 1-{{\left\| x \right\|}^{2}} \right)0 \right) \\
 & \le {{\left\| x \right\|}^{2}}f\left( \left\langle Ay,y \right\rangle  \right)+\left( 1-{{\left\| x \right\|}^{2}} \right)f\left( 0 \right)-{{\left\| x \right\|}^{2}}\left( 1-{{\left\| x \right\|}^{2}} \right)F\left( \left\langle Ay,y \right\rangle -0 \right) \quad \text{(by \eqref{59})}\\
 & \le {{\left\| x \right\|}^{2}}f\left( \left\langle Ay,y \right\rangle  \right)+\left( {{\left\| x \right\|}^{4}}-{{\left\| x \right\|}^{2}} \right)F\left( \left\langle Ay,y \right\rangle  \right) \quad \text{(since $f\left( 0 \right)\le 0$)}\\
 & \le {{\left\| x \right\|}^{2}}\left( \left\langle f\left( A \right)y,y \right\rangle -\left\langle F\left( A-\left\langle Ay,y \right\rangle  \right)y,y \right\rangle  \right)+\left( {{\left\| x \right\|}^{4}}-{{\left\| x \right\|}^{2}} \right)F\left( \left\langle Ay,y \right\rangle  \right) \quad \text{(by \eqref{75})}\\
 & ={{\left\| x \right\|}^{2}}\left( \frac{1}{{{\left\| x \right\|}^{2}}}\left\langle f\left( A \right)x,x \right\rangle -\frac{1}{{{\left\| x \right\|}^{2}}}\left\langle F\left( A-\frac{1}{{{\left\| x \right\|}^{2}}}\left\langle Ax,x \right\rangle  \right)x,x \right\rangle  \right) \\
 &\quad +\left( {{\left\| x \right\|}^{4}}-{{\left\| x \right\|}^{2}} \right)F\left( \frac{1}{{{\left\| x \right\|}^{2}}}\left\langle Ax,x \right\rangle  \right) \\
 & =\left\langle f\left( A \right)x,x \right\rangle -\left\langle F\left( A-\frac{1}{{{\left\| x \right\|}^{2}}}\left\langle Ax,x \right\rangle  \right)x,x \right\rangle +\left( {{\left\| x \right\|}^{4}}-{{\left\| x \right\|}^{2}} \right)F\left( \frac{1}{{{\left\| x \right\|}^{2}}}\left\langle Ax,x \right\rangle  \right).
\end{aligned}\]
This completes the proof.
\end{proof}
\begin{remark}
By taking into account that  $F\left( \cdot \right)$ is a non-negative function, Theorem \ref{76} provides an improvement for {{\cite[Lemma 2.3]{58}}}. More precisely we have
\[\begin{aligned}
   f\left( \left\langle Ax,x \right\rangle  \right)&\le \left\langle f\left( A \right)x,x \right\rangle -\left\langle F\left( A-\frac{1}{{{\left\| x \right\|}^{2}}}\left\langle Ax,x \right\rangle  \right)x,x \right\rangle +\left( {{\left\| x \right\|}^{4}}-{{\left\| x \right\|}^{2}} \right)F\left( \frac{1}{{{\left\| x \right\|}^{2}}}\left\langle Ax,x \right\rangle  \right) \\
 & \le \left\langle f\left( A \right)x,x \right\rangle .
\end{aligned}\]
\end{remark}
\section*{\bf Acknowledgments}
The authors express their gratitude to the anonymous referees for their careful reading and detailed comments which have considerably improved the paper.
\bibliographystyle{alpha}

\begin{thebibliography}{9}
\bibitem{53}
M. Adamek, {\it On a problem connected with strongly convex functions}, Math. Inequal. Appl. {\bf19}(4) (2016), 1287--1293.
\bibitem{5}
S.S. Dragomir, {\it Bounds for the normalized Jensen functional}, Bull. Austral. Math. Soc. {\bf74}(3) (2006), 471--476.
\bibitem{51}
J.B. Hiriart-Urruty, C. Lemar\'echal, {\it Fundamentals of Convex Analysis}, Springer, Berlin (2001).
\bibitem{12}
S. Iveli\'c, A. Matkovi\'c, J. Pe\v cari\'c, {\it On a Jensen-Mercer operator inequality}, Banach J. Math. Anal. {\bf5}(1) (2011), 19--28.
\bibitem{67}
M. Kian, {\it Operator Jensen inequality for superquadratic functions}, Linear Algebra Appl. {\bf456} (2014), 82--87.
\bibitem{17}
W. Liao, J. Wu, J. Zhao, {\it New versions of reverse Young and Heinz mean inequalities with the Kantorovich constant}, Taiwanese J. Math.  {\bf19}(2) (2015),  467--479.
\bibitem{11}
A. Matkovi\'c, J. Pe\v cari\'c, I. Peri\'c, {\it A variant of Jensen's inequality of Mercer's type for operators with applications}, Linear Algebra Appl. {\bf418}(2) (2006), 551--564.
\bibitem{58}
J.S. Matharu, M.S. Moslehian, J.S. Aujla, {\it Eigenvalue extensions of Bohr's inequality}, Linear Algebra Appl. {\bf435}(2) (2011), 270--276.
\bibitem{57}
C.A. McCarthy, {\it ${{c}_{p}}$}, Israel J. Math. {\bf5} (1967), 249--271.
\bibitem{4}
A.McD. Mercer, {\it A monotonicity property of power means}, J. Ineq. Pure and Appl. Math. {\bf3}(3) (2002), Article 40.
\bibitem{2}
A.McD. Mercer, {\it A variant of Jensen's inequality}, J. Ineq. Pure and Appl. Math. {\bf4}(4) (2003), Article 73.
\bibitem{8}
F.-C. Mitroi, {\it Estimating the normalized Jensen functional}, J. Math. Inequal. {\bf5}(4) (2011), 507--521.
\bibitem{editor}
F.-C. Mitroi-Symeonidis, N. Minculete, {\it On the Jensen functional and the strong convexity}, Bull. Malays. Math. Sci. Soc., (2016) DOI:10.1007/s40840-015-0293-z. 
\bibitem{15}
B. Mond, J. Pe\v cari\'c, {\it Convex inequalities in Hilbert space}, Houston J. Math. {\bf19} (1993), 405--420.
\bibitem{1}
N. Merentes, K. Nikodem, {\it Remarks on strongly convex functions}, Aequationes Math. {\bf80}(1-2) (2010), 193--199.
\bibitem{49}
M. Nelson, K. Nikodem, S. Rivas, {\it Remarks on strongly Wright-convex functions}, Annales Polonici Mathematici. {\bf102}(3) (2011),  271--278.
\bibitem{14}
M. Niezgoda, {\it A generalization of Mercer's result on convex functions}, Nonlinear Anal. {\bf71}(7) (2009), 2771--2779.
\bibitem{10}
K. Nikodem, {\it On Strongly Convex Functions and Related Classes of Functions}, Handbook of Functional Equations. Springer New York (2014), 365--405.
\bibitem{54}
K. Nikodem, Z. P\'ales, {\it Characterizations of inner product spaces by strongly convex functions}, Banach J. Math. Anal. {\bf5}(1) (2011), 83--87.
\bibitem{6}
B.T. Polyak, {\it Existence theorems and convergence of minimizing sequences in extremum problems with restrictions}, Soviet Math. Dokl. {\bf7} (1966), 72--75.
\bibitem{13}
A.W. Roberts, D.E. Varberg, {\it Convex functions}, Academic Press, New York-London (1973).
\bibitem{9}
S. Simi\'c, {\it Best possible global bounds for Jensen functional}, Proc. Amer. Math. Soc. {\bf138}(7) (2010), 2457--2462.
\bibitem{16}
H. Zuo, G. Shi, M. Fujii, {\it Refined Young inequality with Kantorovich constant}, J. Math. Inequal. {\bf5}(4) (2011), 551--556.
\end{thebibliography}

\vskip 0.6 true cm

\tiny{$^1$Department of Mathematics, Mashhad Branch, Islamic Azad University, Mashhad, Iran.

{\it E-mail address:} hrmoradi@mshdiau.ac.ir
\vskip 0.3 true cm
$^2$Department of Mathematics, Mashhad Branch, Islamic Azad University, Mashhad, Iran.

{\it E-mail address:} erfanian@mshdiau.ac.ir
\vskip 0.3 true cm

$^3$Department of Mathematics, University of Peshawar, Peshawar, Pakistan.

{\it E-mail address:}  adilswati@gmail.com}
\vskip 0.3 true cm
$^4$Department of Mathematics, University of Bielsko-Biala, Ul. Willowa 2, 43-309 Bielsko-Biala, Poland.

{\it E-mail address:} knikodem@ath.bielsko.pl

\end{document}